\documentclass[12pt]{amsart}
\usepackage{amssymb,amsmath,amsthm,mathrsfs,multirow,enumerate,mathdots}

\usepackage{color}
\usepackage{graphicx} 
\usepackage[all]{xy}
\usepackage{enumitem}
\usepackage{breqn}
\usepackage{tikz}

\numberwithin{equation}{section}

\newtheorem{thm}{Theorem}[section]
\newtheorem{lem}[thm]{Lemma}
\newtheorem{cor}[thm]{Corollary}
\newtheorem{prop}[thm]{Proposition}

\theoremstyle{definition}
\newtheorem{ex}[thm]{Example}

\newtheorem{rmk}[thm]{Remark}

\theoremstyle{definition} \numberwithin{equation}{section}

\newcommand{\G}{\mathrm{G}}
\newcommand{\Cay}{\mathrm{Cay}}
\newcommand{\Spec}{\mathrm{Spec}}
\newcommand{\QFR}{\mathrm{QFR}}
\begin{document}

\title[QFR on unitary Cayley graphs]{Quantum fractional revival on unitary Cayley
graphs over finite commutative rings}

\author{Saowalak Jitngam, Poom Kumam and Songpon Sriwongsa$^*$}

\thanks{*Corresponding Author}

\address{Saowalak Jitngam \\ Department of Mathematics \\ Faculty of Science
	\\ King Mongkut's University of Technology Thonburi (KMUTT) \\ 126 Pracha Uthit Rd., Bang Mod, Thung Khru \\ Bangkok 10140, Thailand}
\email{\tt saowalak.jitn@mail.kmutt.ac.th}

\address{Poom Kumam \\ Department of Mathematics \\ Faculty of Science
	\\ King Mongkut's University of Technology Thonburi (KMUTT) \\ 126 Pracha Uthit Rd., Bang Mod, Thung Khru \\ Bangkok 10140, Thailand}
\email{\tt poom.kum@kmutt.ac.th}

\address{Songpon Sriwongsa \\ Department of Mathematics \\ Faculty of Science
	\\ King Mongkut's University of Technology Thonburi (KMUTT) \\ 126 Pracha Uthit Rd., Bang Mod, Thung Khru \\ Bangkok 10140, Thailand}
\email{\tt  songponsriwongsa@gmail.com}



\begin{abstract}
In this paper, we investigate the existence of quantum fractional revival in unitary Cayley graphs over finite commutative rings with identity. We characterize all finite local rings that permit quantum fractional revival in their unitary Cayley graphs. Additionally, we present results for the case of finite commutative rings, as they can be expressed as products of finite local rings.
\end{abstract}

\date{}

\maketitle

\section{Introduction}
Let $G$ be an undirected simple graph whose vertex set is $V(G) = \{v_1, \ldots, v_n \}$ and its adjacency matrix is denoted by $A_G$. The {\it transition matrix} of $G$ with respect to $A_G$ is defined by:
\begin{equation*}
    H(t) := \exp(itA_G)= \sum_{n\ge 0} \frac{(it)^n}{n!}A_G^n \hspace{4mm} \text{for all} \hspace{4mm} t \ge 0, \text{ where } i = \sqrt{-1}.
\end{equation*} 
Note that $H(t)$ is symmetric and we have $ \overline{H(t)} = H(t)^{-1} $, where $\overline{ \  \cdot  \ }$ denotes the complex conjugate. Furthermore, $H(t)$ is unitary, which implies that $(\overline{H(t)})^T = H(t)^{-1}$. The matrix $H(t)$ is referred to a continuous quantum walk. It is a concept in
quantum physics and quantum theory, where quantum walks are used in various applications. Quantum computers, for instance, leverage quantum walks to perform operations on graphs, including Grover's study \cite{GL} and Farhi and Gutmann's algorithm involving decision trees \cite{FE}. Furthermore, quantum walks on graphs in quantum networks for transferring states were proposed by Bose in 2003 \cite{BS}.


We say that $G$ has \textit{quantum fractional revival} (QFR) from a vertex $v_j$ to a vertex $v_l$ at time $t$ if 
\begin{equation} \label{defn}
    |H(t)_{j, j}|^2 + |H(t)_{j, l}|^2 = 1.
\end{equation}

Let $\vec{e}_s$ denote the standard basis vector in $\mathbb{C}^n$ indexed by the vertex $v_s$ in $G$. It is straightforward to verify that (\ref{defn}) is equivalent to
\[
H(t)\vec{e}_j = \alpha \vec{e}_j + \beta \vec{e}_l,
\]
where $\alpha$ and $\beta$ are complex numbers such that $|\alpha|^2 + |\beta|^2 = 1$. Note that QFR is a generalization of {\it perfect state transfer} (PST) that occurs if $\alpha = 0$. Moreover, if $\beta = 0$, then we say that $G$ is {\it periodic} at the vertex $v_j$. For convenience, when QFR is mentioned, it is assumed to be non-periodic, i.e., $\beta \neq 0$.

QFR was introduced in \cite{RM} as a physical phenomenon in quantum state transfers. It is a process in which, during time evolution, the quantum state cannot be fully transferred to another vertex. Instead, partial state revival occurs at various vertices on a graph, potentially involving the return of portions of the state at different times. 
In physics, the wave function in quantum systems, particularly in the infinite potential well, is analyzed using the fractional revival formalism. This formalism represents the wave function as a superposition of translated copies of the initial wave function, arranged in a parity-conserving manner \cite{AD}. In 2014, Dooley and Spiller focused on studying QFR, multiple-Schr{\"o}dinger-cat states, and quantum carpets in interactions between qubits, which is a key part of quantum physics \cite{DS}. In the present decade, the study of QFR has become increasingly integrated with both physics and mathematics, particularly in the fields of graph theory and spectral analysis.

The mathematical study of quantum state transfer, including PST and QFR, has garnered significant attention in recent years. PST was introduced by Christandl \cite{CM} and has been widely studied since then. In 2012, Godsil provided a comprehensive overview of PST in graphs and related questions \cite{GC2}. For a dihedral group $D_n$, the PST on the Cayley graph
$\text{Cay}(D_n, S)$ was studied in \cite{CX}. The work \cite{CX2} examined weighted Cayley graphs from abelian groups, providing a unified way to describe periodicity and PST. 
Building on this, further investigation for PST in semi-Cayley graphs from abelian groups by proving necessary and sufficient conditions for its occurrence was done in \cite{AM}. Furthermore, the analysis of bi-Cayley graphs from abelian groups, establishing conditions under which PST is possible, was presented in \cite{WS}. Beyond these, the work on PST in Cayley graphs derived from abelian groups with cyclic Sylow-2 subgroups can be seen in \cite{AAS}. In \cite{AM2}, the authors studied PST in Cayley graphs from dicyclic groups, using representations of the dihedral group $D_n$ to formulate conditions for PST. 
 Moreover, necessary and sufficient conditions for PST in Cayley graphs from groups of order $8n$ were discussed in \cite{KA}.
  
In addition to Cayley graphs from groups, PST on graphs from rings has also attracted the interest of many researchers. The study of integral circulant graphs (ICG), offering conditions for identifying which of these graphs allow PST based on their eigenvalues was given in \cite{BM2}. Furthermore, it was revealed that if the two divisors of an ICG are relatively prime, the structure of the graph changes significantly. A similar restrictive result is found in unitary Cayley graphs, where only the complete graph of $2$ vertices $K_2$ and the cycle graph of $4$ vertices $C_4$ permit PST \cite{BM}. The exploration of PST extends beyond these graphs to structures over finite local rings. In \cite{MY}, the authors used the eigenvalues and eigenvectors of a unitary Cayley graph over a finite local ring to determine the conditions under which PST can occur. The work was later expanded to the case of finite commutative ring and gcd-graphs, demonstrating that PST principles apply to an even broader class of graphs \cite{TI}. Various studies have continued to develop and refine these ideas, including those mentioned in \cite{CWC, KVM}, further advancing the understanding of PST in different graph structures.

Later, a study on QFR was conducted, focusing on fractional revival in XX quantum spin chains \cite{GVX}. This research presents models with two parameters that combine isospectral deformations and Para-Krawtchouk polynomials, allowing for both PST and QFR. Building on this work, another study explored graphs that support balanced fractional revival, establishing a connection between quantum walks on hypercubes and extended Krawtchouk spin chains \cite{BPA}. Furthermore, in 2019, it was shown that if a graph $G$ exhibits QFR from a vertex $u$ to a vertex $v$, then it also exhibits QFR from $v$ to $u$ at the same time, reinforcing the symmetry in such quantum systems \cite{CA}. The paper \cite{CA2} explored the conditions needed for QFR in paths and cycles, relating these conditions to state transfer and mixing processes. In graphs, QFR allows one vertex to be represented as a mix of two, enabling entanglement in quantum networks. Additionally, a framework for QFR in spin networks has been proposed, broadening the idea of cospectral vertices and addressing related questions from Chan et al. \cite{CA3}. It is known that PST requires cospectral vertices, while infinite unweighted graphs show QFR between non-cospectral vertices and overlapping pairs \cite{GC3}. 
The study on characterization of QFR between twin vertices in a weighted graph and between the tips of double cones using adjacency, Laplacian, and signless Laplacian matrices was done in \cite{MH}. The authors of \cite{WJ} examined QFR in semi-Cayley graphs over finite abelian groups, outlining the necessary conditions, the need for integrality, and minimum revival time. This includes examples from generalized dihedral and dicyclic groups. Subsequent research also looked into the existence of QFR in Cayley graphs over finite abelian groups \cite{WJ2}. Recently, QFR in unitary Cayley graphs over $\mathbb{Z}_n$ has been analyzed by Soni et al. \cite{SR}.

\smallskip 
The aim of this paper is as follows. Let $R$ be a ring with identity and $R^\times$ a unit group of $R$. The \textit{unitary Cayley graph of} $R$, denoted by $\G_R = \Cay (R,R^\times )$, is a Cayley graph with vertex set $R$ and edge set 
$$
\{\{a,b\} \mid a,b \in R  \text{ and }  a - b \in R^{\times} \}.
$$
Motivated by the aforementioned references, this paper investigates the existence of $\QFR$ on unitary Cayley graphs over finite commutative rings with identity. Specifically, we establish a sufficient and necessary condition for $\G_R$ to have QFR when $R$ is a finite local ring. Furthermore, we present partial results for the case when $R$ is a product of finite local rings. 

\section{Preliminaries}

Let $G$ be an undirected simple graph on $n$ vertices. Then its adjacency matrix $A_G$ is symmetric. Suppose that $A_G$ has $K$ distinct eigenvalues. By a well-known fact from linear algebra, we have the eigenspace orthogonal decomposition as follows:
\[
\mathbb{C}^n=W_1 \oplus W_2 \oplus \cdots \oplus W_K,
\]
 where $W_j$ is an eigenspace corresponding to eigenvalue $\theta_j$ and spanned by orthogonal basis $\{ \Vec{u}_{j_1}, \Vec{u}_{j_2}, \ldots , \Vec{u}_{j_{s_j}} \} $ for some $s_j \in \mathbb{N}$ and for all $j \in \{ 1, 2,\ldots, K \}$ \cite[Theorem 6.3]{FSH}.
 
 For each $j \in \{ 1,2,\ldots,K \}$, let $E_j$ be the orthogonal projection of $\mathbb{C}^n$ into $W_j$. Then the $r$th column of the standard matrix of $E_j$ is given by
     \begin{equation} \label{Ej formula}
         E_j(\Vec{e_r}) =  \sum_{t = 1}^{s_j} \langle \Vec{e}_r,\Vec{u}_{j_t} \rangle \frac{\Vec{u}_{j_t}}{||\Vec{u}_{j_t}||^2}
     \end{equation}
     for all $r \in \{1,2,\cdots ,n\}$, where $\vec{e}_r$ is a standard unit vector whose $r$th entry is $1$ and $0$ elsewhere. 
     
     The Spectral Theorem \cite[Theorem 6.25]{FSH} implies the following properties:
 \begin{enumerate}
      [label=(\roman*)]
     \item  $E_j E_l = \delta_{jl}E_l$ for $1\le j,l \le K$, where $\delta_{jl}$ is the Kronecker delta,
     \item  $\sum_{j = 1}^K E_j = I_n$,
     \item  $\sum_{j = 1}^K \theta_j E_j = A_G$.
     \end{enumerate}
 
Thus we have the identity for the transition matrix of $G$
\begin{align}\label{H(t)}
  H(t)=\sum_{r=1}^{K}\exp(i\theta_r t)E_r.
\end{align}

The following lemma is the key tool for this work.

 \begin{lem}\label{lemmaidem}
   Under the above setting, $G$ has $\QFR$ from a vertex $v_j$ to a vertex $v_l$ at time $t$ if and only if there are constants $\alpha$ and $\beta \neq 0$ with $|\alpha|^2+|\beta|^2=1$ such that 
    $$
      \exp(it\theta_r) E_r \vec{e}_j = \alpha E_r \vec{e}_j + \beta E_r \vec{e}_l,
    $$
    for all $r \in \{1, 2, \ldots, K \}$.
\end{lem}

\begin{proof}
Suppose that $G$ has QFR from vertex $v_j$ to vertex $v_l$ at time $t$. Then there exist $\alpha, \beta \in \mathbb{C}, \beta \neq 0$ such that $|\alpha|^2 + |\beta|^2 = 1$ and
$H(t)\vec{e}_j = \alpha \vec{e}_j + \beta\vec{e}_l$. Equation (\ref{H(t)}) implies $\sum_{s=1}^{K}\exp(i\theta_s t)E_s \vec{e}_j = \alpha \vec{e}_j + \beta\vec{e}_l$. Fixing $r \in \{1, 2, \ldots, K \}$ and multiplying $E_r$ on both sides, we obtain $\exp(i\theta_r t)E_r \vec{e}_j = \alpha E_r \vec{e}_j + \beta E_r \vec{e}_l$ from the properties of $E_r$'s mentioned above.

Conversely, assume that there are constants $\alpha$ and $\beta \neq 0$ such that $|\alpha|^2+|\beta|^2=1$ and
$
    \exp (i\theta_r t)E_r\vec{e}_j = \alpha E_r \vec{e}_j + \beta E_r \vec{e}_l
$
for all $r \in \{1,2,\ldots, K \}$.
Thus,
\begin{align*}
      H(t)\vec{e}_j &= \sum_{r=1}^{K}\exp{(i\theta_r t)}E_r \vec{e}_j \\
    &= \sum_{r=1}^{K}(\alpha E_r \vec{e}_j + \beta E_r \vec{e}_l) \\
    &= \alpha \bigg(\sum_{r=1}^{K} E_r \bigg)\vec{e}_j + \beta \bigg(\sum_{r=1}^{K} E_r\bigg) \vec{e}_l \\
    &= \alpha \vec{e}_j + \beta \vec{e}_l . 
\end{align*}
Therefore, there is QFR from the vertex $v_j$ to the vertex $v_l$ at time $t$.
\end{proof}

The above lemma implies the following result. 

\begin{lem} \label{thm E_r vec e _v}
    If $G$ admits $\QFR$ from a vertex $v_j$ to a vertex $v_l$, then there exist constants $\alpha$ and $\beta \neq 0$ such that $|\alpha|^2+|\beta|^2=1$ and  $$
    E_r \vec{e}_l = \Big( -Re \Big(\frac{\alpha}{\beta}\Big) \pm \sqrt{\Big(Re \Big( \frac{\alpha}{\beta}\Big)\Big)^2 + 1} \Big ) E_r \vec{e}_j
    $$
    for all $r \in \{1, 2, \ldots, K\}$, where $K$ is the number of all distinct eigenvalues of $G$.
\end{lem}

\begin{proof}
    Suppose that there is QFR from vertex $v_j$ to vertex $v_l$ at time t. By Lemma \ref{lemmaidem}, there exist constants $\alpha$ and $\beta \neq 0$ such that $|\alpha|^2 + |\beta|^2 = 1$ and
    $$
        (\exp{(i \theta_r t)} - \alpha ) E_r \vec{e}_j = \beta E_r \vec{e}_l,
    $$
     for all $r \in \{1, 2, \ldots, K\}$.
   Since $\beta \neq 0$, we can let $x = \dfrac{\exp{(i \theta_r t)} - \alpha}{\beta}$. Then we have 
   $$
   1 = |\exp{(i \theta_r t)}|^2 = |x\beta + \alpha|^2
   $$
   and the quadratic polynomial 
    $x^2 + 2 Re(\frac{\alpha}{\beta} )x  -1 = 0$ can be derived. The proof is complete by the quadratic formula. 
\end{proof}

\smallskip

The tensor product of two graphs can be described as follows. Let $G$ and $H$ be two graphs with vertex sets $V(G)$ and $V(H)$, respectively. The vertex set of the tensor graph $G \otimes H$ is the Cartesian product $V(G) \times V(H)$ and two vertices $(a, b)$ and $(a', b')$ in $G \otimes H$ are adjacent if and only if $a$ is adjacent to $a'$ in $G$ and $b$ is adjacent to $b'$ in $H$. The adjacency matrix of $G \otimes H$ is determined by the Kronecker product of the adjacency matrices of $G$ and $H$:
\begin{equation*}
   A_G \otimes A_H = \left [
            \begin{array}{ccc}
               a_{11} A_H & \cdots & a_{1n} A_H \\
\vdots & \ddots & \vdots \\
a_{n1} A_H & \cdots & a_{nn} A_H 
            \end{array}
            \right ],
\end{equation*}
where $a_{jk}$ is the entry of $A_G$ for all $j,k \in \{1,2,\ldots,n \}$. If the eigenvalues of $G$ are $\lambda_1, \ldots, \lambda_n$ and those of $H$ are $\mu_1, \ldots, \mu_m$ (possibly with repetition), then the eigenvalues of $G \otimes H$ are the products $\lambda_i \mu_j$, for $1 \leq i \leq n$ and $1 \leq j \leq m$.

A commutative ring is said to be {\it local} if it has a unique maximal ideal. Note that if $R$ is a local ring with a unique maximal ideal $M$, then $R^\times = R \setminus M$. It is obvious that a field is a local ring with the maximal ideal $\{0\}$.

For a finite commutative ring $R$ with identity, it is well-known that $R$ can be decomposed as $R \cong R_1 \times R_2 \times \cdots \times R_n$, where each $R_i$ is a finite local ring. For the group of units $R^\times$, we have 
$R^\times \cong R_1^\times \times R_2^\times \times \cdots \times R_n^\times$.
The following proposition presents properties of $\G_R$.

\begin{prop}\label{graph}
    \cite[Proposition 2.2]{AR} Let $R$ be a finite commutative ring with identity.
    \begin{enumerate}[label=(\roman*)]
        \item $\G_R$ is a regular graph of degree $|R^ \times|$.
        \item If $R$ is a local ring with maximal ideal $M$, then $\G_R$ is a complete multipartite graph whose partite sets are the cosets of $M$ in $R$. In particular, $\G_R$ is a complete graph if and only if $R$ is a finite field.
        \item If $R \cong R_1 \times R_2 \times \cdots \times R_t$ is a product of finite local rings, then $\G_R \cong \bigotimes_{i=1}^{t}\G_{R_i}.$
    \end{enumerate}
\end{prop}

If $\theta_1, \theta_2,\ldots,\theta_K$ are the eigenvalues of a graph $G$ with multiplicities $m_1, m_2,\ldots, m_K$, respectively, then the spectrum of $G$ is expressed as 
$$
\Spec G =\left(
    \begin{array}{cccc}
    \theta_1 & \theta_2 & \ldots & \theta_K \\
    m_1 & m_2 & \ldots & m_K
    \end{array}
    \right). 
$$

\begin{prop}\label{specG_R} \cite[Proposition 2.1]{KD} 
    Let $R$ be a finite local ring with a unique maximal ideal $M$ of size $m$. Then
    \begin{equation*}
        \Spec \G_R = \left(
    \begin{array}{ccc}
    |R^\times| & -m & 0\\
   1 & \frac{|R|}{m}-1 & \frac{|R|}{m}(m-1)
    \end{array}
    \right) .
    \end{equation*}
    In particular, if $\mathbb{F}_q$ is the finite field with $q$ elements, then
    \begin{equation*}
       \Spec \G_{\mathbb{F}_q} = \left(
    \begin{array}{cc}
    q-1 & -1 \\
   1 & q-1
    \end{array}
    \right).
    \end{equation*}
\end{prop}

\section{Main results} 

We divide this section into two subsections. First, we consider a unitary Cayley graph over a finite local ring. The second subsection is devoted to the graph $\G_R$ when $R$ is a product of finite local rings.  

\subsection{Over a finite local ring} 

Let $R$ be a finite local ring with a unique maximal ideal $M$ of size $m$.
By Proposition \ref{graph}(ii), the adjacency matrix of $\G_R$ can be described as follows:
\begin{equation*}
    A_{\G_R}=
            \left [
            \begin{array}{ccccc}
               0_{m}  & J_{m} & J_{m} & \cdots & J_{m} \\
               J_{m}  & 0_{m} & J_{m} & \cdots & J_{m} \\
              J_{m}  & J_{m} & 0_{m} & \cdots & J_{m} \\
              \vdots & \vdots & \vdots & \ddots & \vdots \\
              J_{m}  & J_{m} & J_{m} & \cdots & 0_{m}
            \end{array}
            \right ]  ,   
\end{equation*}
where $0_m$ (resp. $J_m$) is the zero square matrix (resp. one) of order $m$.
By Proposition \ref{specG_R},  $\G_R$ has eigenvalues $\theta_1 = |R^\times|, \theta_2 = -m$ and $\theta_3 = 0$ with multiplicities $1,\frac{|R|}{m}-1$ and $\frac{|R|}{m}(m-1)$, respectively.

Denote $\Vec{0}_k = 
\begin{bmatrix}
    0 \\
    0 \\
    \vdots \\
    0
\end{bmatrix}_{k \times 1}$ and $ \Vec{1}_k = \begin{bmatrix}
    1 \\
    1 \\
    \vdots \\
    1
\end{bmatrix}_{k \times 1}$.
Then 
\begin{enumerate}
    \item the eigenspace corresponding to $\theta_1 = |R^\times|$ is spanned by $\vec{1}_{|R|}$,
   
    \item the eigenspace corresponding to $\theta_2 = -m$ is spanned by 
    \[
    \left\{  \begin{bmatrix}
        \frac{1}{s} \vec{1}_{sm} \\
        -\vec{1}_m \\
        \vec{0}_{\big(\frac{|R|}{m} - 1 - s \big)m}
    \end{bmatrix}
    \;\middle|\; s = 1, 2, \ldots, \frac{|R|}{m}-1
    \right\},
    \]
    
    \item the eigenspace corresponding to $\theta_3 = 0$ is spanned by all columns of the following block matrix
    \begin{equation*}
    \begin{bmatrix}
        W &&&\\
        &W&&\\
        &&\ddots& \\
        &&&W
    \end{bmatrix}_{|R|\times\frac{|R|}{m}(m-1)}
\end{equation*}
where
\begin{equation*}
    W = \begin{bmatrix}
        1 & 1 & \cdots & 1 \\
        \omega & \omega^2 & \cdots & \omega^{m-1}\\
        \omega^2 & \omega^4 & \cdots & \omega^{2(m-1)}\\
        \vdots & \vdots & \ddots & \vdots \\
        \omega^{m-1} & \omega^{2(m-1)} & \cdots & \omega^{(m-1)(m-1)}
    \end{bmatrix}_{m\times(m-1)}
\end{equation*}
and $\omega=\exp(\frac{2i\pi}{m})$.
\end{enumerate}

From the direct calculation using the formula (\ref{Ej formula}), the orthogonal projections $E_j$ into the eigenspace corresponding to the eigenvalue $\theta_j$ are as follows: 
\begin{enumerate}
     \item $E_1 = \frac{1}{|R|}J_{|R|}$,
     
        \item $E_2 = \begin{bmatrix}
            \Vec{w}_1 & \Vec{w}_2&\cdots &\Vec{w}_{|R|-1} &\Vec{w}_{|R
            |}
        \end{bmatrix}$, where
            \begin{enumerate} 
            [label=(\roman*)]
                \item $\Vec{w}_s=\sum_{l=1}^{\frac{|R|}{m}-1}\frac{\Vec{u}_l}{(l+1)m}$,
                \item 
    $\Vec{w}_{km+s}=\sum_{l=k}^{\frac{|R|}{m}-1}\frac{\Vec{u}_l}{(l+1)m}-\frac{\Vec{u}_k}{(k+1)m}$, 
                \item 
                $\Vec{w}_{|R|-m+s} = \bigg( \frac{m-|R|}{|R|m} \bigg)\Vec{u}_{\frac{|R|}{m}-1}$, 
            \end{enumerate}
        
        for all $s \in \{1,2,\ldots,m\}$, $k \in \{1,2, \ldots , \frac{|R|}{m}-2\}$ and $\Vec{u}_l$ is the $l$th column of $A_2$,
        \item $E_3=\frac{1}{m} \begin{bmatrix}
            N & & &\\
              & N& &\\
              & &\ddots&\\
              & & &N 
        \end{bmatrix}_{|R| \times |R|}$, \\
        where $N = \begin{bmatrix}
            m-1&-1&-1&&-1\\
            -1&-1&-1&&m-1\\
    \vdots&\vdots&\vdots& \iddots &\vdots\\
             -1&-1&m-1&&-1\\
            -1&m-1&-1&&-1
        \end{bmatrix}_{m\times m}$.
\end{enumerate}

\medskip
A necessary condition for QFR to occur is provided in the following proposition, which relates to the concept of strongly cospectral vertices. 

\begin{prop}
\label{plus-minus}
    Let $R$ be a finite local ring. If $\QFR$ occurs in $\G_R$ from a vertex $v_j$ to a vertex $v_l$, then $E_r \vec{e}_l = \pm E_r \vec{e}_j$ for all $r \in \{1,2,3\}$.
\end{prop}
\begin{proof}
Suppose that $\G_R$ has QFR from a vertex $v_j$ to a vertex $v_l$.
    From Lemma \ref{thm E_r vec e _v}, there exist constants $\alpha$ and $\beta \neq 0$ such that $|\alpha|^2+|\beta|^2=1$ and 
    \begin{equation*}
    E_r \vec{e}_l = \Big( -Re \Big(\frac{\alpha}{\beta}\Big) \pm \sqrt{\Big(Re \Big( \frac{\alpha}{\beta}\Big)\Big)^2 + 1} \Big) E_r \vec{e}_j
    \end{equation*} for $r = 1, 2, 3$.
    When $r = 1$, the equation implies that 
    $$
    -Re \Big(\frac{\alpha}{\beta}\Big) \pm \sqrt{\Big(Re \Big( \frac{\alpha}{\beta}\Big)\Big)^2 + 1} = 1.
    $$
    Then it can be solved that $Re \Big(\dfrac{\alpha}{\beta}\Big) = 0$. The proof is complete.
\end{proof}

To allow QFR to occur in $\G_R$, the finite local ring $R$ must satisfy the following necessary condition.

\begin{thm}\label{m12}
    Let $R$ be a finite local ring with a unique maximal ideal $M$ of size $m$. If $\QFR$ occurs in $\G_R$, then $m$ is $1$ or $2$.
\end{thm}
\begin{proof}
    Suppose that $\QFR$ occurs in $\G_R$ from a vertex $v_j$ to a vertex $v_l$.
    By Proposition \ref{plus-minus}, we have $E_3 \vec{e}_l = \pm E_3 \vec{e}_j$. Thus, the $l$th column of $E_3$ is equal to $\pm$ the $j$th column of $E_3$, which implies that $m-1 = -1, 0$ or $1$. Since $m = |M| > 0$, we have the theorem.
\end{proof}

If $R$ is a finite field, we have a complete characterization as follows.

\begin{thm}\label{q=2}
    Let $\mathbb{F}_q$ be the finite field with $q$ elements. Then $\QFR$ occurs in  $\G_{\mathbb{F}_q}$ if and only if $q = 2$.
\end{thm}
\begin{proof}
    Assume that $q = 2$. Then $A_{\G_R}= \left [
            \begin{array}{cc}
               0  & 1  \\
               1  & 0
            \end{array}
            \right ]$ 
            and 
            $$
            H(t) = \exp{(itA_{G_R})}
            = \begin{bmatrix}
                \cos{t} & i \sin{t} \\
                i \sin{t} & \cos{t}
            \end{bmatrix}
            $$ 
            for all $t \ge 0$. Note that $|\cos{t}|^2 + |i \sin{t}|^2 = 1$ for all $t \ge 0$. Thus, $\G_{\mathbb{F}_q}$ admits QFR. 
            
            To prove the other direction, we assume that $q \ge 3$. Similarly to the proof of \cite[Theorem 2.3]{MY}, we can use the above matrix $E_2$ to argue that $E_2 \vec{e}_l \neq \pm E_2 \vec{e}_j$ for all $1 \le j < l \leq q$, which contradicts Proposition \ref{plus-minus}.
\end{proof}

For the case of a finite local ring, we obtain the following characterization.

\begin{thm} \label{mainlocal}
     Let $R$ be a finite local ring with a unique maximal ideal $M$. Then $\G_R$ has $\QFR$ if and only if $R$ is $\mathbb{F}_2, \mathbb{Z}_4$ or $\mathbb{Z}_2[x]/(x^2)$.
\end{thm}
\begin{proof}
   Assume that $\G_R$ has $\QFR$. By Theorem \ref{m12}, $|M| = 1$ or $2$. If $|M| = 1$, then $R$ is a finite field and Theorem \ref{q=2} forces $R = \mathbb{F}_2$. If $|M| = 2$, it follows from \cite{GN} that $R$ is either $\mathbb{Z}_4$ or $\mathbb{Z}_{2} [ x ] /(x^2)$.

    Suppose that $R = \mathbb{F}_2, \mathbb{Z}_4$ or $\mathbb{Z}_{2} [ x ] /(x^2)$. Then by \cite[Theorem 2.4]{MY}, $\G_R$ admits PST, so it has QFR. 
\end{proof}

We provide examples to support our above results.

\begin{ex}
   Consider the following unitary Cayley graphs over $\mathbb{Z}_n$, where $n = 2, 3$ and $4$. 
\begin{enumerate}
    \item Note that $\G_{\mathbb{Z}_2}$ is the path graph $P_2$. 
   
    \vspace{0.5cm}
    \begin{center}
            \begin{tikzpicture}[scale=1, every node/.style={circle, draw, fill=black, inner sep=2pt}]
  \node[label=below:$v_1$] (v1) at (0,0) {};
  \node[label=below:$v_2$] (v2) at (2,0) {};
  \draw (v1) -- (v2);
\end{tikzpicture}
    \end{center}
    It has $\QFR$ between the two vertices, as shown in Theorem \ref{q=2}.

    \item The graph $\G_{\mathbb{Z}_3}$ is the complete graph $K_3$.

    \vspace{0.5cm}
     \begin{center}
 \begin{tikzpicture}[scale=1, every node/.style={circle, draw, fill=black, inner sep=1.8pt}]
  \node (v0) at (90:1)  [label=above:$v_1$] {};
  \node (v1) at (210:1) [label=below left:$v_2$] {};
  \node (v2) at (330:1) [label=below right:$v_3$] {};

  \draw (v0) -- (v1);
  \draw (v1) -- (v2);
  \draw (v2) -- (v0);
\end{tikzpicture}
     \end{center}
     Consider the orthogonal projection
$E_2$ into the eigenspace corresponding to the eigenvalue $\theta_2 = -1$:
\[
E_2 = \frac{1}{3}
\begin{bmatrix}
    2 & -1 & -1 \\
    -1 & 2 & -1 \\
    -1 & -1 & 2
\end{bmatrix}.
\]
Then $E_2 \vec{e}_l \neq \pm E_2 \vec{e}_j$ for all $j, l \in \{1, 2, 3\}$. By Proposition \ref{plus-minus}, $\QFR$ does not occur between two vertices.

\item The graph $\G_{\mathbb{Z}_4}$ is a $4$-cycle $C_4$.

\vspace{0.5cm}
     \begin{center}
\begin{tikzpicture}[scale=1, every node/.style={circle, draw, fill=black, inner sep=1.8pt}]
  \node (v0) at (0,1)   [label=above:$v_1$] {};
  \node (v1) at (1,0)   [label=right:$v_2$] {};
  \node (v2) at (0,-1)  [label=below:$v_3$] {};
  \node (v3) at (-1,0)  [label=left:$v_4$] {};

  \draw (v0) -- (v1);
  \draw (v1) -- (v2);
  \draw (v2) -- (v3);
  \draw (v3) -- (v0);
\end{tikzpicture}
\end{center}
Its adjacency matrix is given by
\[
A_{\G_{\mathbb{Z}_4}} = \begin{bmatrix}
            0 & 1 & 0 & 1 \\
               1 & 0 & 1 & 0 \\
               0 & 1 & 0 & 1 \\
               1 & 0 & 1 & 0  
    \end{bmatrix}
\]
and the transition matrix is equal to 
\begin{align*}
          \begin{bmatrix}
\cos^2(t) & i\sin(t)\cos(t) & -\sin^2(t) & i\sin(t)\cos(t) \\
i\sin(t)\cos(t) & \cos^2(t) & i\sin(t)\cos(t) & -\sin^2(t) \\
-\sin^2(t) & i\sin(t)\cos(t) & \cos^2(t) & i\sin(t)\cos(t) \\
i\sin(t)\cos(t) & -\sin^2(t) & i\sin(t)\cos(t) & \cos^2(t)
\end{bmatrix}.
    \end{align*}
Recall that the periodic cases are ruled out. Note that there is no time $t$ satisfying 
\[
|\cos^2(t)|^2 + |i \sin(t) \cos (t)|^2 = 1 \text{ and } |\cos^2(t)|^2 \neq 1.
\]
However, the conditions 
\[
|\cos^2(t)|^2 + |- \sin^2(t)|^2 = 1 \text{ and } |\cos^2(t)|^2 \neq 1 
\]
imply $t = \frac{\pi}{2} + k \pi$, where $k \in \mathbb{Z}$. Therefore, $\QFR$, in fact PST, only occurs at these times $t$ between vertices $v_1$ and $v_3$.
\end{enumerate}

\end{ex}

\subsection{Over a finite commutative ring}

The results of PST concerning unitary Cayley graphs over finite commutative rings directly lead to the following theorem.

\begin{thm}
    Let $\mathbb{F}_{2^{r_1}}, \ldots, \mathbb{F}_{2^{r_s}}$ be the finite fields with $2^{r_1},\ldots,2^{r_s}$ elements, respectively. For $R = \mathbb{F}_2, \mathbb{Z}_4$ or $\mathbb{Z}_2[x] / (x^2)$, the graph $\G_R \otimes \G_{\mathbb{F}_{2^{r_1}}} \otimes 
    \cdots 
    \otimes \G_{\mathbb{F}_{2^{r_s}}}$ has $\QFR$. 
\end{thm}
\begin{proof}
   By \cite[Theorem 2.5]{TI}, all these graphs have PST, so they have QFR.
\end{proof}

The characterization of finite commutative rings allowing PST to occur on their unitary Cayley graphs has been completed in \cite[Theorem 2.5]{TI}. The authors showed that a finite commutative ring $R$ does not contain an odd characteristic local ring as its component if $\G_R$ exhibits PST. This implies that $\G_{\mathbb{Z}_6}$ does not possess PST. However, the following example shows that QFR can occur in this graph.

\begin{ex}
The graph $\G_{\mathbb{Z}_6}$ is a $6$-cycle $C_6$ (a hexagon).

    \vspace{0.5 cm}
    \begin{center}
       \begin{tikzpicture}[scale=1, every node/.style={circle, draw, fill=black, inner sep=1.8pt}]
  \foreach \i in {1,...,6} {
    \node (v\i) at (60*\i:1) [label=60*\i:$v_{\i}$] {};
  }

  \foreach \i [evaluate=\i as \j using {int(mod(\i,6)+1)}] in {1,...,6} {
    \draw (v\i) -- (v\j);
  }
\end{tikzpicture}
    \end{center}
     Its adjacency matrix is given by
    $$
    A_{\G_{\mathbb{Z}_6}} =
    \begin{bmatrix}
       0 & 1 & 0 & 0 & 0 & 1 \\
               1 & 0 & 1 & 0 & 0 & 0 \\
               0 & 1 & 0 & 1 & 0 & 0 \\
               0 & 0 & 1 & 0 & 1 & 0 \\
               0 & 0 & 0 & 1 & 0 & 1 \\
               1 & 0 & 0 & 0 & 1 & 0  
    \end{bmatrix},
            $$ 
            and there are four distinct eigenvalues, $\theta_1 = -2, \theta_2 = -1, \theta_3 = 1$ and $\theta_4=2$ with multiplicities $1, 2, 2$ and $1$, respectively. 
             The eigenspaces are spanned, respectively, by the columns of the following matrices: 
            \begin{align*}
                A_1 &= \left [
            \begin{array}{c}
               1   \\
               -1  \\
               1 \\
               -1 \\
               1\\
               -1
            \end{array}
            \right ] , \hspace{4mm} A_2 = \left [
            \begin{array}{cc}
               1 & 1   \\
               \omega^2 & \omega^4 \\
               \omega^4 & \omega^2 \\
               1  & 1\\
               \omega^2 & \omega^4 \\
               \omega^4 & \omega^2
            \end{array}
            \right ], \\ 
            A_3 &= \left [
            \begin{array}{cc}
               1 & 1   \\
               \omega & \omega^5 \\
               \omega^2 & \omega^4 \\
               -1  & -1\\
               \omega^4 & \omega^2 \\
               \omega^5 & \omega
            \end{array}
            \right ] , \hspace{2mm} \text{and} 
            \hspace{2mm} A_4 = \left [
            \begin{array}{c}
               1   \\
               1  \\
               1 \\
               1 \\
               1\\
               1
            \end{array}
            \right ],
            \end{align*} where $\omega= \exp(\frac{i\pi}{3})$.
   Moreover, we obtain the orthogonal projection $E_j$ on the eigenspace belonging to $\theta_j$ for each $j \in \{1,2,3,4 \}$, as follows:
   \begingroup
\allowdisplaybreaks
   \begin{align*}
       E_1 &= \frac{1}{6}\left [
            \begin{array}{cccccc}
               1  & -1 & 1 & -1 & 1 & -1  \\
               -1  & 1 & -1 & 1 & -1 & 1 \\
               1  & -1 & 1 & -1 & 1 & -1  \\
               -1  & 1 & -1 & 1 & -1 & 1 \\
               1  & -1 & 1 & -1 & 1 & -1  \\
               -1  & 1 & -1 & 1 & -1 & 1
            \end{array}
            \right ], \\  
            E_2 &= \frac{1}{6}\left [
            \begin{array}{cccccc}
               2  & -1 & -1 & 2 & -1 & -1  \\
               -1  & -1 & 2 & -1 & -1 & 2 \\
               -1  & 2 & -1 & -1 & 2 & -1  \\
               2  & -1 & -1 & 2 & -1 & -1 \\
               -1  & -1 & 2 & -1 & -1 & 2  \\
               -1  & 2 & -1 & -1 & 2 & -1
            \end{array}
            \right ], \\ 
            E_3 &= \frac{1}{6}\left [
            \begin{array}{cccccc}
               2  & 1 & -1 & -2 & -1 & 1  \\
               1  & -1 & -2 & -1 & 1 & 2 \\
               -1  & -2 & -1 & 1 & 2 & 1  \\
               -2  & -1 & 1 & 2 & 1 & -1 \\
               -1  & 1 & 2 & 1 & -1 & -2  \\
               1  & 2 & 1 & -1 & -2 & -1
            \end{array}
            \right ],  
            \text{ and } \\ \hspace{2mm} 
            E_4 &= \frac{1}{6}\left [
            \begin{array}{cccccc}
               1  & 1 & 1 & 1 & 1 & 1  \\
               1  & 1 & 1 & 1 & 1 & 1 \\
               1  & 1 & 1 & 1 & 1 & 1 \\
               1  & 1 & 1 & 1 & 1 & 1 \\
               1  & 1 & 1 & 1 & 1 & 1 \\
               1  & 1 & 1 & 1 & 1 & 1
            \end{array}
            \right ]. 
   \end{align*} 
   \endgroup
   By Lemma \ref{lemmaidem} and choosing $\alpha = -\frac{1}{2}$ and $\beta = -i\frac{\sqrt{3}}{2}$, QFR occurs in $\G_{\mathbb{Z}_6}$ from $v_1$ to $v_4$ at time $t=\frac{2\pi}{3}$.
\end{ex}

The above example highlights a difference between PST and QFR, motivating further study.

\medskip
Let $R$ be a finite commutative ring with identity. Then 
\[
R \cong R_1 \times R_2 \times \ldots \times R_n,
\]
where $R_j$ is a finite local ring with a unique maximal ideal $M_j$ of size $m_j$ for all $j \in \{1,2\ldots,n\}$. Set $m = m_1 m_2 \ldots m_n$. Since $R_j$ is a finite local ring, the graph $\G_{R_j}$ is a complete multipartite graph whose partite sets are the cosets of $M_j$ in $R_j$.
Let $K_t$ be the complete graph on $t$ vertices. Then the adjacency matrix of $\G_{R_j}$ can be expressed as $A_{\G_{R_j}} = J_{m_j} \otimes A_{K_{\frac{|R_j|}{m_j}}}$ for all $j \in \{1,2,\ldots,n\}$ which implies 
\begin{align*}
       A_{\G_{R}} &= \bigotimes_{j=1}^{n}A_{\G_{R_j}}= \bigotimes_{j=1}^{n}\bigg( J_{m_{j}} \otimes A_{K_{\frac{|R_i|}{m_j}}} \bigg) \\
    &= \bigotimes_{j=1}^{n}J_{m_j} \otimes \bigotimes_{j=1}^{n}A_{K_{\frac{|R_j|}{m_j}}} = J_m \otimes A_G,
\end{align*}
where $G$ is the tensor graph $K_{\frac{|R_1|}{m_1}} \otimes K_{\frac{|R_2|}{m_2}} \otimes \cdots \otimes K_{\frac{|R_n|}{m_n}}$.

For each $j \in \{1,2,\ldots,n\}$, let $m_j' = \frac{|R_j|}{m_j}$ and consider the orthogonal projection on the eigenspace corresponding into the eigenvalue $m_j' - 1$ for the graph $K_{m_j'}$ which is given by 
$$
\frac{1}{m_j'} J_{m_j'} =
\frac{1}{m_j'}\begin{bmatrix}
        1 & 1& \cdots &1\\
        1 & 1& \cdots &1 \\
\vdots&\vdots&\ddots&\vdots\\
        1 & 1& \cdots &1
    \end{bmatrix}.
$$

Focusing on $J_m$, its eigenvalues are $m$ of multiplicity one and $0$ of multiplicity $m - 1$. Let $E_1$ be the orthogonal projection into the eigenspace corresponding to $m$ and $E_2$ be the orthogonal projection into the eigenspace corresponding to $0$. It follows that
\begin{equation*}
    E_1 = \frac{1}{m}\begin{bmatrix}
        1 & 1& \cdots &1\\
        1 & 1& \cdots &1 \\
        \vdots&\vdots&\ddots&\vdots\\
        1 & 1& \cdots &1
    \end{bmatrix} \hspace{2mm} \text{and} \hspace{2mm} E_2 = \frac{1}{m}\begin{bmatrix}
        m-1 & -1  & \cdots & -1\\
        -1 & m-1 & \cdots & -1\\
        \vdots & \vdots & \ddots &\vdots\\
        -1 & -1 & \cdots& m-1
    \end{bmatrix}.
\end{equation*}
Using the properties of the tensor product, the adjacency matrix $A_{\G_R}$ admits an orthogonal projection of the form 
\[
P_1 = E_1 \otimes \bigotimes_{j = 1}^n \frac{1}{m_j'} J_{m_j'} =
\frac{1}{|R|}J_{|R|} = 
\frac{1}{|R|}\begin{bmatrix}
        1 & 1& \cdots &1\\
        1 & 1& \cdots &1 \\
\vdots&\vdots&\ddots&\vdots\\
        1 & 1& \cdots &1
    \end{bmatrix}, 
\]
which is for the eigenvalue $m \cdot \prod_{j = 1} ^n (m_j' - 1)$.

\begin{prop}
 \label{plus-minus commu}
    Let $R$ be a finite commutative ring with identity. If $\QFR$ occurs in $\G_R$ from a vertex $v_j$ to a vertex $v_l$, then $P_r \vec{e}_l = \pm P_r \vec{e}_j$ for every orthogonal projection $P_r$ into the eigenspace corresponding to $\theta_r$ of $\G_R$.
\end{prop}
\begin{proof}
We can argue using similar arguments from the proof of Proposition \ref{plus-minus} and apply $P_1$ to show that the constant $Re\big( \frac{\alpha}{\beta}\big)$ is $0$.
\end{proof}

Note that $A_{\G_R}$ also has an orthogonal projection into the eigenspace corresponding to the eigenvalue $0$, given by
\[
P_2 = E_2 \otimes I_{m_1'} \otimes \cdots \otimes I_{m_n'} = E_2 \otimes \sum E_{m'_1} \otimes \cdots \otimes \sum E_{m'_n},
\]
where $I_{m_j'}$ is the identity matrix of size $m_j'$ and $\sum E_{m'_j}$ is the sum of all orthogonal projections of $K_{m'_j}$.

\begin{thm}\label{m=1,2}
    Let $R$ be a finite commutative ring with identity. Under the above notations, if $\QFR$ occurs in $\G_R$, then $m = 1$ or $2$. 
\end{thm}
\begin{proof}
  Assume that QFR occurs in $\G_R$ from a vertex $v_j$ to a vertex $v_l$. If $m > 2$, then we would have $P_2 \vec{e}_l \neq \pm P_2 \vec{e}_j$, which contradicts Proposition \ref{plus-minus commu}.
\end{proof}

The above theorem leads to the following corollary. 

\begin{cor}\label{evenonly}
    Let $R$ be a finite commutative ring with identity. If $\QFR$ occurs in $\G_R$, then $|R|$ is even. 
\end{cor}
\begin{proof}
    As mentioned above, we can write $R \cong R_1 \times R_2 \times \ldots \times R_t$, where $R_j$ is a finite local ring with a unique maximal ideal $M_j$ of size $m_j$ for all $j \in \{1,2\ldots,t\}$. Set $m = m_1 m_2 \ldots m_t$. Suppose that $|R|$ is odd and $\G_R$ admits QFR. By Theorem \ref{m=1,2}, $m = m_1 = m_2 = \dots = m_t = 1$. Thus, $A_{\G_R} = K_{|R_1|} \otimes K_{|R_2|} \otimes \cdots \otimes K_{|R_t|}$. It has an orthogonal projection of the form 
    \[
    B = \bigotimes_{j = 1}^t I_{|R_j|} - \frac{1}{|R_j|} J_{|R_j|}.
    \]
    In fact, this matrix is for the eigenvalue $(-1)^t$ of $A_{\G_R}$. Note that for all $j = 1, 2, \ldots, t$, $R_j$ is a finite field of odd characteristic and $B \vec{e}_l \neq \pm B \vec{e}_k$ for any distinct $k, l \in \{1, 2, \ldots, |R| \}$. Again, by Proposition \ref{plus-minus commu}, this is a contradiction.
\end{proof}

\smallskip

Given graphs $X$ and $Y$ with the adjacency matrix of $X$ having the spectral decomposition $A_X = \sum_{r = 1}^d \theta_r E_r$, by Lemma 4.2 of \cite{CG15}, we have 
\[
H_{X \times Y}(t) = \sum_{r = 1}^d E_r \otimes H_Y(\theta_r t). 
\]
This leads to the following results.

\begin{thm} \label{QFR finite com ring}
     Let $\mathbb{F}_q$ be the finite field with an odd prime power $q$. The graph $\G_{\mathbb{F}_q} \otimes \G_{\mathbb{F}_2}$ has $\QFR$ at time $t = \frac{2\pi}{q}$.
\end{thm} 
\begin{proof}
  Recall that $H_{\G_{\mathbb{F}_2}}(t) = 
  \begin{bmatrix}
      \cos({t}) & i \sin({t}) \\
      i \sin({t}) & \cos({t})
  \end{bmatrix}$ and the eigenvalues of $\G_{\mathbb{F}_q}$ are $\theta_1 = q - 1$ and $\theta_2 = -1$. Choosing $t = \frac{2\pi}{q}$ implies that for $r = 1, 2$,
  $$
  H_{\G_{\mathbb{F}_2}}(\theta_r t) = 
  \begin{bmatrix}
      \cos(\theta_r \frac{2\pi}{q}) & i \sin(\theta_r \frac{2\pi}{q}) \\
      i \sin(\theta_r \frac{2\pi}{q}) & \cos(\theta_r \frac{2\pi}{q}) 
  \end{bmatrix}
  =
  \begin{bmatrix}
      \cos(\frac{2\pi}{q}) & -i \sin( \frac{2\pi}{q}) \\
      -i \sin( \frac{2\pi}{q}) & \cos(\frac{2\pi}{q}) 
  \end{bmatrix}.
  $$
    Thus 
    \[
H_{\G_{\mathbb{F}_q} \times \G_{\mathbb{F}_2}}(t) = \sum_{r = 1}^2 E_r \otimes H_{\G_{\mathbb{F}_2}}(\theta_r t) = I_q \otimes \begin{bmatrix}
      \cos(\frac{2\pi}{q}) & -i \sin( \frac{2\pi}{q}) \\
      -i \sin( \frac{2\pi}{q}) & \cos(\frac{2\pi}{q}) 
  \end{bmatrix}. 
\]
Since $|\cos(\frac{2\pi}{q})|^2 + |-i \sin( \frac{2\pi}{q})|^2 = 1$, the graph $\G_{\mathbb{F}_q} \otimes \G_{\mathbb{F}_2}$ admits $\QFR$ between $v_1$ and $v_2$.
\end{proof}

To compare with Theorem \ref{QFR finite com ring}, it is natural to ask what happens if the second component is changed from $\G_{\mathbb{F}_2}$ to $\G_{\mathbb{Z}_4}$ or $\G_{\mathbb{Z}_{2} [ x ] /(x^2)}$. To the best of the authors’ knowledge, the resulting calculations become significantly more complicated. Nevertheless, we can present a partial result as follows. 

\begin{thm} \label{QFR finite com ring Z4}
      Let $\mathbb{F}_q$ be the finite field with an odd prime power $q$. If $R = \mathbb{Z}_4$ or $\mathbb{Z}_{2} [ x ] /(x^2)$, then $\QFR$ cannot occur in
      $\G_{\mathbb{F}_q} \otimes \G_R$ at any time $\tau \in \{\frac{k \pi}{q} \mid k \in \mathbb{Z} \}$.
\end{thm} 
\begin{proof}
   Recall that
    \begin{align*}
           H_{\G_R}(t) 
           = \begin{bmatrix}
\cos^2(t) & i\sin(t)\cos(t) & -\sin^2(t) & i\sin(t)\cos(t) \\
i\sin(t)\cos(t) & \cos^2(t) & i\sin(t)\cos(t) & -\sin^2(t) \\
-\sin^2(t) & i\sin(t)\cos(t) & \cos^2(t) & i\sin(t)\cos(t) \\
i\sin(t)\cos(t) & -\sin^2(t) & i\sin(t)\cos(t) & \cos^2(t)
\end{bmatrix}.
    \end{align*} 
Note that $\theta_1 = q - 1$ and $\theta_2 = -1$ are the eigenvalues of $\G_{\mathbb{F}_q}$. Let $\tau = \frac{k \pi}{q}$, where $k \in \mathbb{Z}$. Then $H_{\G_R}(\theta_1 \tau) = H_{\G_R}(\theta_2 \tau)$ which implies 
    \begin{align*}
       \mathcal{H} := H_{\G_{\mathbb{F}_q} \times \G_R}(\tau) = \sum_{r = 1}^2 E_r \otimes H_{\G_R}(\theta_r \tau) 
 = I_q \otimes H_{\G_R}(- \tau)
 \end{align*}
 and $|\mathcal{H}_{j, j}|^2 + |\mathcal{H}_{j, l}|^2 \neq 1$ if $|\mathcal{H}_{j, j}| \neq 1$. This completes the proof.
\end{proof}

\begin{rmk}
    In the proof of the above theorem, we observe that $H_{\G_R}(\theta_1 t) = H_{\G_R}(\theta_2 t)$ if and only if $t = \frac{k\pi}{q}$, where $k \in \mathbb{Z}$.
\end{rmk}

We conclude the paper by presenting the results regarding the existence of $\QFR$ on unitary Cayley graphs of certain small finite commutative rings.

\begin{ex}
    We consider unitary Cayley graphs over $\mathbb{Z}_n$ for $n \le 29$ except $n = 12, 20, 28$. By Corollary \ref{evenonly}, if $n$ is odd, then $\G_{\mathbb{Z}_n}$ does not have QFR. Thus, only even positive integers $n$ need to be investigated.  
    
    Theorem \ref{mainlocal} implies that the graph $\G_{\mathbb{Z}_n}$ has QFR for $n=2$ and $4$ while for $n = 8$ and $16$, $\G_{\mathbb{Z}_n}$ does not possess QFR. According to Theorem \ref{m=1,2}, for $n = 18$ and $24$, QFR does not occur in $\G_{\mathbb{Z}_n}$. 
    For $n = 6, 10, 14, 22$ and $26$, $\G_{\mathbb{Z}_n}$ admits QFR by Theorem \ref{QFR finite com ring}. 
\end{ex}

\begin{rmk}
    Our approach in this paper does not apply to certain cases, such as the graphs $\G_{\mathbb{Z}_{12}}, \G_{\mathbb{Z}_{20}}, \G_{\mathbb{Z}_{28}}$ and $\G_{\mathbb{Z}_{30}}$. Further developments and new techniques are required for a complete characterization for finite commutative rings that permit QFR to occur in their unitary Cayley graphs.
\end{rmk}

\bigskip

\section*{Acknowledgments}
The first author was supported by the Petchra Pra Jom Klao Master’s Degree Research Scholarship from King Mongkut’s University of Technology Thonburi (Grant No. 11/2566). This research was funded by King Mongkut’s University of Technology Thonburi (KMUTT), Thailand Science Research and Innovation (TSRI), and the National Science, Research, and Innovation Fund (NSRF) under Fiscal Year 2024 Grant No. FRB670073/0164.



\begin{thebibliography}{99}
\bibitem{GL} 
Grover LK. A fast quantum mechanical algorithm for database search. Proc. Annu. ACM Symp. Theory Comput. 1996;212--219.

\bibitem{FE} Farhi E, Gutmann E. Quantum computation and decision trees. Phys. Rev. A. 1998;58(2):915.

\bibitem{BS} Bose S. Quantum communication through an unmodulated spin chain. Phys. Rev. Lett. 2003;91(20):207901.

\bibitem{RM} Rohith M, Sudheesh C. Visualizing revivals and fractional revivals in a Kerr medium using an optical tomogram. Phys. Rev. A. 2015;92(5):053828.

\bibitem{AD} Aronstein DL, Stroud  CR. Fractional wave-function revivals in the infinite square well. Phys. Rev. A. 1997;55(6):4526.

\bibitem{DS} Dooley S, Spiller TP. Fractional revivals, multiple-Schr{\"o}dinger-cat states, and quantum carpets in the interaction of a qubit with N qubits. Phys. Rev. A. 2014;90(1):012320.

\bibitem{CM} Christandl M, Datta N, Ekert A, Landahl AJ. Perfect state transfer in quantum spin networks. Phys. Rev. Lett. 2004;92(18):187902.

\bibitem{GC2} Godsil C. State transfer on graphs. Discrete Math. 2012;312(1):129--147.

\bibitem{CX} Cao X, Chen B, Ling S. Perfect state transfer on Cayley graphs over dihedral groups: The non-normal case. Electron. J. Comb. 2020;27(2):P2.28.

\bibitem{CX2} Cao X, Feng K, Tan YY. Perfect state transfer on weighted abelian Cayley graphs. Chin. Ann. Math. Ser. B. 2021;42(4):625--642.

\bibitem{AM} Arezoomand M. Perfect state transfer on semi-Cayley graphs over abelian groups. Linear Multilinear Algebra. 2023;71(14):2337--2353.

\bibitem{WS} Wang S, Feng T.  Perfect state transfer on bi-Cayley graphs over abelian groups. Discrete Math. 2023;346(6):113362.

\bibitem{AAS}  {\'A}rnad{\'o}ttir AS, Godsil C. On state transfer in Cayley graphs for abelian groups. Quantum Inf. Process. 2022;22(1):8.

\bibitem{AM2} Arezoomand M, Shafiei F, Ghorbani M. Perfect state transfer on Cayley graphs over the dicyclic group. Linear Algebra Appl. 2022;639:116--134.

\bibitem{KA} Kalita A, Bhattacharjya B. Perfect state transfer on Cayley graphs over a non-abelian group of order $8 n$. arXiv:2405.02122. (2024).

\bibitem{BM2} Ba{\v{s}}i{\'c} M, Petkovi{\'c} MD. Perfect state transfer in integral circulant graphs of non-square-free order. Linear Algebra Appl. 2010;433(1):149--163.

\bibitem{BM} Ba{\v{s}}i{\'c} M, Petkovi{\'c} MD, Stevanovi{\'c} D. Perfect state transfer in integral circulant graphs. Appl. Math. Lett. 2009;22(7):1117--1121.

\bibitem{MY} Meemark Y, Sriwongsa S. Perfect state transfer in unitary Cayley graphs over local rings. Trans. Comb. 2014;3(4):43--54.

\bibitem{TI} Thongsomnuk I,  Meemark Y. Perfect state transfer in unitary Cayley graphs and gcd-graphs. Linear Multilinear Algebra. 2019;67(1):39--50.

\bibitem{CWC} Cheung WC, Godsil C. Perfect state transfer in cubelike graphs. Linear Algebra Appl. 2011;435(10):2468--2474.

\bibitem{KVM} Kendon VM, Tamon C. Perfect state transfer in quantum walks on graphs. J. Comput. Theor. Nanosci. 2011;8(3):422--433.

\bibitem{GVX} Genest VX, Vinet L, Zhedanov A. Quantum spin chains with fractional revival. Ann. Phys. 2016;371:348--367.

\bibitem{BPA} Bernard PA, Chan A, Loranger {\'E}, Tamon C, Vinet L. A graph with fractional revival. Phys. Lett. A. 2018;382(5):259--264.

\bibitem{CA} Chan A, Coutinho G, Tamon C, Vinet L, Zhan H. Quantum fractional revival on graphs. Discret. Appl. Math. 2019;269:86--98.

\bibitem{CA2} Chan A, Coutinho G, Tamon C, Vinet L, Zhan H. Fractional revival and association schemes. Discrete Math. 2020;343(11):112018.

\bibitem{CA3} Chan A, Coutinho G, Drazen W, Eisenberg O, Godsil C, Kempton M, Lippner G, Tamon C, Zhan H. Fundamentals of fractional revival in graphs. Linear Algebra Appl. 2022;655:129--158.

\bibitem{GC3} Godsil C, Zhang X. Fractional revival on non-cospectral vertices. Linear Algebra Appl. 2022;654:69--88.

\bibitem{MH} Monterde H. Fractional revival between twin vertices. Linear Algebra Appl. 2023;676:25--43.

\bibitem{WJ} Wang J, Wang L, Liu X. Fractional revival on semi-Cayley graphs over abelian groups. Linear Multilinear Algebra. 2025;73(8):1611--1633.

\bibitem{WJ2} Wang J, Wang L, Liu . Fractional revival on Cayley graphs over abelian groups. Discrete Math. 2024;347(12):114218.



\bibitem{SR} Soni R, Choudhary N, Singh NP. Quantum Fractional Revival in Unitary Cayley Graphs. Lobachevskii J. Math. 2025;46(4):1922--1928.

\bibitem{FSH} Friedberg SH, Insel AJ, Spence LE. Linear Algebra, 4th edn. Prentice Hall, New Jersey; 2003.


\bibitem{AR} Akhtar R, Boggess M, Jackson-Henderson T, Jim{\'e}nez I, Karpman R, Kinzel A, Pritikin D. On the unitary Cayley graph of a finite ring. Electron. J. Comb. 2009;16(1):R117.

\bibitem{KD} Kiani D, Aghaei MMH, Meemark Y, Suntornpoch B. Energy of unitary Cayley graphs and gcd-graphs. Linear Algebra Appl. 2011;435(6):1336--1343.

\bibitem{GN} Ganesan N. Properties of rings with a finite number of zero divisors. Math. Ann. 1964;157(3):215--218.

\bibitem{CG15} Coutinho G, Godsil C. Perfect state transfer in products and covers of
graphs. Linear Multilinear Algebra. 2015;64(2):235-246. 


\end{thebibliography}
\end{document}